\documentclass[12pt, reqno]{amsart}
\usepackage{amssymb,latexsym,amsmath,amscd,amsthm,graphicx, color}
\usepackage[all]{xy}
\usepackage{pgf,tikz}
\usepackage{mathrsfs}
\usetikzlibrary{arrows}
\usepackage[left=0.6 in, top=0.6 in, right=0.6 in, bottom=0.3 in]{geometry}
\makeatletter
\newcommand{\setword}[2]{%
  \phantomsection
  #1\def\@currentlabel{\unexpanded{#1}}\label{#2}%
}
\makeatother

\usepackage[linkcolor=blue, urlcolor=blue, citecolor=blue,
colorlinks, bookmarks]{hyperref}

\definecolor{uuuuuu}{rgb}{0.26666666666666666,0.26666666666666666,0.26666666666666666}
\definecolor{xdxdff}{rgb}{0.49019607843137253,0.49019607843137253,1.}
\definecolor{ffqqqq}{rgb}{1.,0.,0.}
\definecolor{ffqqqq}{rgb}{1.,0.,0.}
\definecolor{ffxfqq}{rgb}{1.,0.4980392156862745,0.}

\definecolor{uuuuuu}{rgb}{0.26666666666666666,0.26666666666666666,0.26666666666666666}
\definecolor{qqwuqq}{rgb}{0.,0.39215686274509803,0.}
\definecolor{zzttqq}{rgb}{0.6,0.2,0.}
\definecolor{xdxdff}{rgb}{0.49019607843137253,0.49019607843137253,1.}
\definecolor{qqqqff}{rgb}{0.,0.,1.}
\definecolor{cqcqcq}{rgb}{0.7529411764705882,0.7529411764705882,0.7529411764705882}
\definecolor{sqsqsq}{rgb}{0.12549019607843137,0.12549019607843137,0.12549019607843137}

\theoremstyle{plain}

\newtheorem{theorem1}[subsubsection]{Theorem}

\newtheorem{lemma}[subsection]{Lemma}

\newtheorem{prop}[subsection]{Proposition}
\newtheorem{prop1}[subsubsection]{Proposition}
\newtheorem{lemma1}[subsubsection]{Lemma}

\theoremstyle{definition}
\newtheorem{defi1}[subsubsection]{Definition}

\newtheorem{exam}[subsubsection]{Example}

\newtheorem{remark}[subsection]{Remark}
\newtheorem{remark1}[subsubsection]{Remark}

\newtheorem{conj}[subsection]{Conjecture}
\newtheorem{open}[subsection]{Open}

\newtheorem{note}[subsection]{Note}

\newcommand{\uu}{\cup}
\newcommand{\ii}{\cap}

\newcommand{\sci}{\subset}

\newcommand{\es}{\emptyset}
\newcommand{\set}[1]{\{#1\}}

\newcommand{\ga}{\alpha}
\newcommand{\gb}{\beta}

\newcommand{\tbf}{\textbf}
\newcommand{\tit}{\textit}

\newcommand{\C}[1]{\mathcal{#1}}
\newcommand{\D}[1]{\mathbb{#1}}

\newcommand{\te}{\text}

\newcommand{\nd}{\noindent}

\begin{document}
\nd To appear, Arabian Journal of Mathematics
\title{Quantization for the mixtures of uniform distributions on connected and disconnected line segments}

\address{School of Mathematical and Statistical Sciences\\
The University of Texas Rio Grande Valley\\
1201 West University Drive\\
Edinburg, TX 78539-2999, USA.}
\email{$^1$ashabarua@vt.edu}
\email{\{$^2$gustavo.fernandez03, $^3$ashley.gomez06,  $^5$mrinal.roychowdhury\}@utrgv.edu}
\email{$^4$ogla.lopez7457@gmail.com}
 \author{$^1$Asha Barua}
\author{$^2$Gustavo Fernandez}
\author{$^3$Ashley Gomez}
\author{$^4$Ogla Lopez}
 \author{$^5$Mrinal Kanti Roychowdhury}

\subjclass[2010]{60E05, 94A34.}
\keywords{Mixed distribution, uniform distribution, optimal sets of $n$-means, quantization error}

\date{}
\maketitle

\pagestyle{myheadings}\markboth{A. Barua, G. Fernandez, A. Gomez, O. Lopez, and M.K. Roychowdhury }{Quantization for the mixtures of uniform distributions on connected and disconnected line segments}

\begin{abstract}

In this paper, we have studied various mixed distributions generated by two uniform distributions: first, where the supports are two connected line segments, and second, where the supports are two disconnected line segments. For these mixed distributions, we have determined the optimal sets of $n$-means and the corresponding $n$th quantization errors for all positive integers $n $. The methods developed in this paper can be applied more generally to investigate optimal quantization for any mixed distribution
$
P := pP_1 + (1 - p)P_2,
$
where $P_1$ and $P_2$ are arbitrary probability distributions supported on either connected or disconnected line segments, and $(p, 1 - p)$ is any probability vector with $0 < p < 1$.
\end{abstract}

\section{Introduction}

Let $\D R^d$ denote the $d$-dimensional Euclidean space equipped with a metric $\|\cdot\|$ compatible with the Euclidean topology. Let $P$ be a Borel probability measure on $\D R^d$ and $\ga$ be a locally finite subset of $\D R^d$, i.e., intersection of $\ga$ with any bounded subset of $\D R^d$ is finite. This implies that $\ga$ is countable and closed. 
Then, $\int \min_{a \in \ga} \|x-a\|^2 dP(x)$ is often referred to as the \tit{cost,} or \tit{distortion error} for $\ga$ with respect to the probability measure $P$, and is denoted by $V(P; \ga)$.  Write
$\C D_n:=\set{\ga \sci \D R^d : 1\leq \te{card}(\ga)\leq n}$.  Then, $\inf\set{V(P; \ga) : \ga \in \C D_n}$ is called the \tit{$n$th quantization error} for the probability measure $P$, and is denoted by $V_n:=V_n(P)$. A set $\ga$ for which the infimum occurs and contains no more than $n$ elements is called an \tit{optimal set of $n$-points}. It is known that for a Borel probability measure $P$,  if its support contains at least $n$ elements and $\int \| x\|^2 dP(x)$ is finite, then an optimal set of $n$-points always has exactly $n$-elements (see \cite{AW, GKL, GL1, GL2}). 
For some recent work in the direction of optimal sets of $n$-points, one can see \cite{CR, DR1, DR2, GL3, L1, R1, R2, R3, R4, R5, R6, RR1, RS}. Optimal quantization has broad application in engineering and technology (see \cite{GG, GN, Z}).

Let us now state the following proposition (see \cite{GL2, GG}).
\begin{prop} \label{prop0}
Let $\ga$ be an optimal set of $n$-points for $P$, and $a\in \ga$. Then,

$(i)$ $P(M(a|\ga))>0$, $(ii)$ $ P(\partial M(a|\ga))=0$, $(iii)$ $a=E(X : X \in M(a|\ga))$,
where $M(a|\ga)$ is the Voronoi region of $a\in \ga, $ i.e.,  $M(a|\ga)$ is the set of all elements $x$ in $\D R^d$ which are closest to $a$ among all the elements in $\ga$, and $\partial M(a|\ga)$ represents the boundary of the Voronoi region $M(a|\ga)$.
\end{prop}
By the above proposition, we see that in unconstrained quantization, the elements in an optimal set of $n$-points are the conditional expectations in their own Voronoi regions. Because of this fact, in unconstrained quantization, an optimal set of $n$-points is termed an optimal set of 
$n$-means.
 
Mixed distributions are an exciting new area for optimal quantization. For any two Borel probability measures $P_1$ and $P_2$, and $p\in (0, 1)$, if $P:=p P_1+(1-p)P_2$, then the probability measure $P$ is called the \tit{mixture} or the \tit{mixed distribution} generated by the probability measures $(P_1, P_2)$ associated with the probability vector $(p, 1-p)$.

\subsection{Delineation}
In this paper, first, we assume that $P_1$ and $P_2$ are two uniform distributions on the two connected line segments $J_1:=[0, 1]$ and $J_2:=[1, 2]$, respectively. Then, for the mixed distribution $P:=pP_1+(1-p)P_2$ with support the closed interval $[0, 2]$, in Subsection~\ref{sec4} for $p=\frac 1{5}$ and in Subsection~\ref{sec5} for $p=\frac 13$, we determine the optimal sets of $n$-means and the $n$th quantization errors for all $n\in \D N$. Next, assume that $P_1$ and $P_2$ are two uniform distributions on the two disconnected line segments $J_1:=[0, \frac 13]$ and $J_2:=[\frac 23, 1]$, respectively. Then, for the mixed distribution $P:=pP_1+(1-p)P_2$ with support the union of the closed intervals $[0, \frac 13]$ and $[\frac 23, 1]$, in Subsection~\ref{subsec44} for $p=\frac 1{100}$ and in Subsection~\ref{sec55} for $p=\frac 25$ and $p=\frac 1{1000}$, we determine the optimal sets of $n$-means and the $n$th quantization errors for all $n\in \D N$. In addition, in this paper,  we give a conjecture Conjecture~\ref{conj1}, and two open problems Open~\ref{op1} and Open~\ref{op2}. Under a conjecture in Subsection~\ref{sec66}, we give a partial answer of the open problem Open~\ref{op2}.

\subsection{Significance}
The paper provides a systematic approach for computing the optimal sets of $n$-means and corresponding quantization errors for any mixture of two distributions supported on connected or disconnected line segments 
This is useful in modeling real-world data that comes from multiple sources, such as sensor fusion or audio signals from multiple environments.

\begin{remark}
In the sequel, there are some decimal numbers which are rational approximations of some real numbers.
\end{remark}

\section{Mixed distribution $P$ with support the closed interval $[0, 2]$}
\subsection{Basic preliminaries} \label{sec3}

Let $P_1$ and $P_2$ be two uniform distributions, respectively, on the intervals given by
 \[J_1:=[0, 1], \te{ and } J_2:=[1, 2].\]
 Let $f_1$ and $f_2$ be their respective density functions. Then,
\[f_1(x)=\left \{\begin{array} {cc}
1 & \te{ if } x\in [0, 1], \\
0 & \te{otherwise};
\end{array}
\right. \te{ and } f_2(x)=\left \{\begin{array} {cc}
1 & \te{ if } x\in [1, 2], \\
0 & \te{otherwise}.
\end{array}
\right.\]

  Let us consider the mixed distribution $P:=pP_1+(1-p)P_2$, where $0<p<1$. Notice that $P$ has support the closed interval $[0, 2]$. By $E(X)$, it is meant the expectation of a random variable $X$ with probability distribution $P$, and $V(X)$ represents the variance of $X$. By $\ga_n:=\ga_n(P)$, we denote an optimal set of $n$-means with respect to the probability distribution $P$, and $V_n:=V_n(P)$ represents the corresponding quantization error for $n$-means.
Since $f_1$ and $f_2$ are the density functions for the probability distributions $P_1$ and $P_2$, respectively, we have
 \[dP_1(x)=P_1(dx)=f_1(x)dx=dx,   \te{ and } dP_2(x)=P_2(dx)=f_2(x) dx=dx,\]
 where $d$ stands for differential. 

\begin{lemma}\label{lemma0}
Let $P$ be the mixed distribution defined by $P=p P_1+(1-p) P_2$. Then, $E(X)=\frac{1}{2} (3-2 p)$, and $V(X)=-p^2+p+\frac{1}{12}$.
\end{lemma}

\begin{proof}
We have
\[E(X)=\int x dP=p \int x d(P_1(x))+(1-p)\int x d(P_2(x))=p \int_{J_1}x \, dx+(1-p) \int_{J_2}  x \, dx\]
yielding $E(X)=\frac{1}{2} (3-2 p)$, and
\[V(X)=\int (x-E(X))^2 dP=p \int(x-E(X))^2 d(P_1(x))+(1-p) \int(x-E(X))^2 d(P_2(x)),\]
implying
$V(X)=-p^2+p+\frac{1}{12}$, and thus, the proposition is yielded.
\end{proof}

\begin{note} Lemma~\ref{lemma0} implies that the optimal set of one-mean is the set $\set{\frac{1}{2} (3-2 p)}$, and the corresponding quantization error is the variance $V:=V(X)$ of a random variable with probability distribution $P$. For a subset $J$ of $\D R$ with $P(J)>0$,
by $P(\cdot|_J)$, we denote the conditional probability given that $J$ is occurred, i.e., $P(\cdot|_J)=P(\cdot\ii J)/P(J)$, in other words, for any Borel subset $B$ of $\D R$ we have $P(B|_J)=\frac{P(B\ii J)}{P(J)}$.
\end{note}

Let us now give the following proposition.

\begin{prop} \label{prop31}
Let $P$ be a Borel probability measure on $\D R$ such that $P$ is uniformly distributed over a closed interval $[a, b]$ with a constant density function $f$ such that $f(x)=t$ for all $x\in [a, b]$, where $t\in \D R$. Then, the optimal set $\ga_n(P(\cdot|_{[a, b]}))$ of $n$-means and the corresponding quantization error $V_n(P(\cdot|_{[a,b]}))$ of $n$-means for the probability distribution $P(\cdot|_{[a, b]})$ are, respectively, given by
\[\ga_n(P(\cdot|_{[a, b]})):=\Big\{a+\frac{(2j-1)(b-a)}{2n} : 1\leq j\leq n\Big\}, \te{ and } V_n(P(\cdot|_{[a,b]}))=\frac {(b-a)^3t}{12n^2}.\]
\end{prop}
\begin{proof}
Let $\ga_n:=\set{a_1<a_2<\cdots<a_n}$ be an optimal set of $n$-means for the probability distribution $P$ with a constant density function $f$ on $[a, b]$ such that $f(x)=t$ for all $x\in [a, b]$, where $t\in \D R$. Then, proceedings analogously as \cite[Theorem~2.1.1]{RR2}, we can show that  $a_j=a+\frac{(2j-1)(b-a)}{2n}$ implying
\[\ga_n(P(\cdot|_{[a, b]}))=\Big\{a+\frac{(2j-1)(b-a)}{2n} : 1\leq j\leq n\Big\}.\]
 Notice that the probability density function is constant, and the Voronoi regions of the elements $a_j$ for $1\leq j\leq n$ are of equal lengths. This yields the fact that the distortion errors due to each $a_j$ are equal. Hence, the $n$th quantization error is given by
\begin{align*}
V_n(P(\cdot|_{[a, b]}))&=\int \min_{a\in \ga_n(P(\cdot|_{[a, b]})} (x-a)^2dP=n t\int_a^{\frac 12(a_1+a_2)}  (x-a_1)^2  dx\\
&=nt \int_{a}^{a+\frac {b-a}{n}}(x-(a+\frac{b-a}{2n}))^2 dx
\end{align*}
implying
\[V_n(P(\cdot|_{[a, b]}))=\frac {(b-a)^3t}{12n^2}.\]
 Thus, the proof of the proposition is complete.
\end{proof}

For $n\geq 3$, let $\ga_n$ be an optimal set of $n$-means such that $\ga_n$ contains elements from both $[0, 1]$ and $(1, 2]$. Then, there exist two positive integers $k:=k(n)$ and $m:=m(n)$ such that $\te{card}(\ga_n\ii [0, 1])=k$, and $\te{card}(\ga_n\ii (1, 2])=m$. Let $\ga_n\ii [0, 1]:=\set{a_1<a_2<\cdots<a_k}$ and $\ga_n\ii (1, 2]=\set{b_1<b_2<\cdots<b_m}$. Then, $k+m=n$. Moreover, notice that the following two cases can happen: either $\frac 12(a_k+b_1)\leq 1$, or $1\leq \frac 12(a_k+b_1)$.
Let $V1(k, m)$ be the $n$th quantization error when  $\frac{a_k+b_1}{2}\leq 1$, and $V2(k, m)$ be the $n$th quantization error when  $1\leq \frac{a_k+b_1}{2}$. The following propositions give the optimal sets of $n$-means and the $n$th quantization errors for all $n\geq 3$:

\begin{prop} \label{prop71}
Let $k\geq 2$ and $m=1$. Then, if $\frac 12(a_k+b_1)\leq 1$, we have $a_j=\frac{(2j-1)(a_k+b_1)}{4k}$ for $1\leq j\leq k$, and $b_1=E(X : X\in [\frac 12(a_k+b_1), 2])$. On the other hand, if $1\leq \frac 12(a_k+b_1)$, we have $a_j=\frac{(2j-1)(a_{k-1}+a_k)}{4 (k-1)}$ for $1\leq j\leq (k-1)$, $a_k=E(X : X\in [\frac{1}{2} (a_{k-1}+a_k), \frac 12(a_k+b_1)])$, and
$b_1= E(X : X\in [\frac 12(a_k+b_1), 2])$.
The quantization errors for $n$-means are given by
\begin{align*}
 V1(k, 1)&=p \frac{ (a_k+b_1 ){}^3}{96 k^2}+ p \int_{\frac{1}{2}  (a_k+b_1 )}^{1}  (x-b_1 ){}^2 \, dx+(1-p)\int_{1}^{2}  (x-b_1 ){}^2 \, dx, \te{ and } \\
V2(k, 1)&=p\frac{(a_{k-1}+a_k){}^3}{96 (k-1)^2}+p \int_{\frac{1}{2} (a_{k-1}+a_k)}^{1} (x-a_k){}^2 \, dx+(1-p)\int_{1}^{\frac{1}{2} (a_k+b_1)} (x-a_k){}^2 \, dx\\
&\qquad \qquad  +(1-p)\int_{\frac{1}{2} (a_k+b_1)}^{2} (x-b_1){}^2 \, dx.
\end{align*}
 \end{prop}
\begin{proof}
If $\frac{a_k+b_1}{2}\leq 1$, then $a_1, a_2, \cdots, a_{k}$ are uniformly distributed over the closed interval $[0, \frac{a_k+b_1}{2}]$; on the other hand,
if $1 \leq \frac{a_k+b_1}{2}$, then $a_1, a_2, \cdots, a_{k-1}$ are uniformly distributed over the closed interval $[0, \frac{a_{k-1}+a_k}2]$.
 Thus, by Proposition~\ref{prop0} and Proposition~\ref{prop31}, the expressions for $a_j$ and $b_j$ can be obtained. With the help of the formula given in Proposition~\ref{prop31}, the quantization errors are also obtained as routine.
\end{proof}

\begin{prop} \label{prop72}
Let $k=1$ and $m\geq 2$. Then, if $\frac 12(a_1+b_1)\leq 1$, we have $a_1=E(X : X\in [0, \frac 12(a_1+b_1)]$, $b_1=E(X : X\in [\frac 12(a_1+b_1), \frac 12(b_1+b_2)])$, and
\[b_{1+j}=
\frac 12(b_1+b_2)+\frac{(2j-1)}{2(m-1)}(2-\frac 12(b_1+b_2)) \te{ for } 1\leq j\leq m-1.\]
On the other hand, if $1\leq \frac 12(a_1+b_1)$, we have $a_1=E(X : X\in [0, \frac 12(a_1+b_1)]$, and
 \[b_{j}=
\frac 12(a_1+b_1)+\frac{(2j-1)}{2m}(2-\frac 12(a_1+b_1)) \te{ for } 1\leq j\leq m.\]
The quantization errors for $n$-means are given by
\begin{align*}
 V1(1, m)&=
 p\int_0^{\frac{1}{2}  (a_1+b_1 )}  (x-a_1 ){}^2 \, dx+p \int_{\frac{1}{2}  (a_1+b_1 )}^{1}  (x-b_1 ){}^2 \, dx+(1-p)\int_{1}^{\frac{1}{2}  (b_1+b_2 )}  (x-b_1 ){}^2 \, dx\\
& \qquad \qquad  +(1-p) \frac{(4-(b_1+b_2))^3}{96(m-1)^2}, \te{ and } \\
 V2(1, m)&=p \int_0^{1} (x-a_1){}^2 \, dx+(1-p)\int_{1}^{\frac{1}{2} (a_1+b_1)} (x-a_1){}^2 \, dx
 +(1-p)\frac{(4-(a_1+b_1)){}^3}{96 m^2}.
 \end{align*}
\end{prop}

\begin{proof}
If $\frac 12(a_1+b_1)\leq 1$, then $b_2, b_3, \cdots, b_m$ are uniformly distributed over the closed interval $[\frac 12(b_1+b_2), 2]$, and so by Proposition~\ref{prop0} and Proposition~\ref{prop31} the expressions for $a_j$ and $b_j$, and the corresponding quantization error are obtained.
Likewise, if $1\leq \frac 12(a_1+b_1)$, by Proposition~\ref{prop0} and Proposition~\ref{prop31}, the expressions for $a_j$ and $b_j$, and the corresponding quantization error are  obtained.
\end{proof}

\begin{prop} \label{prop73}
Let $k\geq 2$ and $m\geq 2$. Then, if $\frac 12(a_k+b_1)\leq 1$, we have $a_j=\frac{(2j-1)(a_k+b_1)}{4k}$ for $1\leq j\leq k$,
$b_1=E(X : X\in [\frac 12(a_k+b_1), \frac 12(b_1+b_2)])$, and
\[
b_{1+j}=
\frac 12(b_1+b_2)+\frac{(2j-1)}{2(m-1)}(2-\frac 12(b_1+b_2)) \te{ for } 1\leq j\leq m-1. \]
On the other hand, if $1\leq \frac 12(a_k+b_1)$, we have $a_j=\frac{(2j-1)(a_{k-1}+a_k)}{4 (k-1)}$ for $1\leq j\leq (k-1)$, $a_k=E(X : X\in [\frac{1}{2} (a_{k-1}+a_k), \frac 12(a_k+b_1)])$, and
\[
b_{j}=
\frac 12(a_k+b_1)+\frac{(2j-1)}{2m}(2-\frac 12(a_k+b_1)) \te{ for } 1\leq j\leq m.\]
The quantization errors for $n$-means are given by
\begin{align*}
 V1(k, m)&=p\frac{ (a_k+b_1 ){}^3}{192 k^2}+p \int_{\frac{1}{2}  (a_k+b_1 )}^{1}  (x-b_1 ){}^2 \, dx+(1-p)\int_{2}^{\frac{1}{2} (b_1+b_2)}  (x-b_1 ){}^2 \, dx\\
 &\qquad \qquad +(1-p)\frac{(4-(b_1+b_2))^3}{96(m-1)^2}, \te{ and }\\
 V2(k, m)&=p\frac{(a_{k-1}+a_k){}^3}{192 (k-1)^2}+p\int_{\frac{1}{2} (a_{k-1}+a_k)}^{1} (x-a_k){}^2 \, dx+(1-p)\int_{1}^{\frac{1}{2} (a_k+b_1)} (x-a_k){}^2 \, dx\\
 &\qquad \qquad+(1-p)\frac{\left(4- (a_k+b_1)\right){}^3}{96 m^2}.
 \end{align*}
 \end{prop}

\begin{proof}
Notice that Proposition~\ref{prop73} is a mixture of Proposition~\ref{prop71} and Proposition~\ref{prop72}, and thus the proof follows in the similar way.
\end{proof}

\subsection{Quantization for the mixed distribution $P$ when $p=\frac 1{5}$} \label{sec4}

Notice that if $p=\frac 1{5}$, then $E(X)=\frac{13}{10}$ and $V(X)=\frac{73}{300}$, i.e., the optimal set of one-mean for $p=\frac 15$ is $\set{\frac{13}{10}}$ and the corresponding quantization error is $V(X)=\frac{73}{300}$.

\begin{prop1}   \label{prop1}
The optimal set of two-means is  $\set{\frac{11}{16}, \frac{25}{16}}$ with quantization error $V_2=\frac{317}{3840}$.
\end{prop1}
\begin{proof}
Let $\ga:=\set{a_1, a_2}$ be an optimal set of two-means. Recall that the mixed distribution $P$ has support $[0, 2]$. Since the elements in an optimal set are the conditional expectations in their own Voronoi regions, without any loss of generality, we can assume that $0<a_1<a_2<2$. The following cases can arise:

\tit{Case~1. $0<a_1<a_2\leq 1$. }

Since the boundary of the Voronoi region is $\frac 12(a_1+a_2)$, we have the distortion error as
\begin{align*}
&\int \min_{a\in \ga}(x-a)^2 dP=\frac{1}{5} \int_0^{\frac{1}{2} \left(a_1+a_2\right)} \left(x-a_1\right){}^2 \, dx+\frac{1}{5} \int_{\frac{1}{2} \left(a_1+a_2\right)}^1 \left(x-a_2\right){}^2 \, dx+\frac{4}{5} \int_1^2 \left(x-a_2\right){}^2 \, dx\\
&=\frac{1}{60} \left(3 a_1^3+3 a_2 a_1^2-3 a_2^2 a_1-3 a_2^3+60 a_2^2-156 a_2+116\right),
\end{align*}
the minimum value of which is $\frac{37}{135}=0.274074$ and it occurs when $a_1=\frac 13, $ and $a_2=1$.

\tit{Case~2. $0<a_1\leq 1 <a_2< 2$. }

In this case, the boundary $\frac 12(a_1+a_2)$ of the Voronoi regions of $a_1$ and $a_2$ satisfies one of the following two conditions: either $\frac {a_1+a_2}{2}\leq 1$, or $1\leq \frac {a_1+a_2}{2}$. Consider the following two subcases:

\tit{Subcase~1. $0<a_1<\frac {a_1+a_2}{2}\leq 1 <a_2< 2$. }

In this subcase, the distortion error is given by
\begin{align*}
&\int \min_{a\in \ga}(x-a)^2 dP=\frac{1}{5} \int_0^{\frac{1}{2} \left(a_1+a_2\right)} \left(x-a_1\right){}^2 \, dx+\frac{1}{5} \int_{\frac{1}{2} \left(a_1+a_2\right)}^1 \left(x-a_2\right){}^2 \, dx+\frac{4}{5} \int_1^2 \left(x-a_2\right){}^2 \, dx\\
&=\frac{1}{60} \left(3 a_1^3+3 a_2 a_1^2-3 a_2^2 a_1-3 a_2^3+60 a_2^2-156 a_2+116\right),
\end{align*}
the minimum value of which is $\frac{1}{12}=0.0833333$ and it occurs when $a_1=\frac 12, $ and $a_2=\frac 32$.

\tit{Subcase~2. $0<a_1\leq 1\leq \frac {a_1+a_2}{2}<a_2< 2$. }

In this subcase, the distortion error is given by
\begin{align*}
&\int \min_{a\in \ga}(x-a)^2 dP=\frac{1}{5} \int_0^1 \left(x-a_1\right){}^2 \, dx+\frac{4}{5} \int_1^{\frac{1}{2} \left(a_1+a_2\right)} \left(x-a_1\right){}^2 \, dx+\frac{4}{5} \int_{\frac{1}{2} \left(a_1+a_2\right)}^2 \left(x-a_2\right){}^2 \, dx\\
&=\frac{1}{15} \left(3 a_1^3+3 \left(a_2-3\right) a_1^2-3 \left(a_2^2-3\right) a_1-3 a_2^3+24 a_2^2-48 a_2+29\right).
\end{align*}
the minimum value of which is $\frac{317}{3840}=0.0825521$ and it occurs when $a_1=\frac{11}{16}, $ and $a_2=\frac{25}{16}$.

\tit{Case~3. $1\leq a_1<a_2<2$. }

In this case the distortion error is given by
\begin{align*}
&\int \min_{a\in \ga}(x-a)^2 dP\\
&=\frac{1}{5} \int_0^1 \left(x-a_1\right){}^2 \, dx+\frac{4}{5} \int_1^{\frac{1}{2} \left(a_1+a_2\right)} \left(x-a_1\right){}^2 \, dx+\frac{4}{5} \int_{\frac{1}{2} \left(a_1+a_2\right)}^2 \left(x-a_2\right){}^2 \, dx\\
&=\frac{1}{15} \left(3 a_1^3+3 \left(a_2-3\right) a_1^2-3 \left(a_2^2-3\right) a_1-3 a_2^3+24 a_2^2-48 a_2+29\right),
\end{align*}
the minimum value of which is $\frac{13}{135}=0.0962963$ and it occurs when $a_1=1, $ and $a_2=\frac 53$.

Comparing the distortion errors obtained in all the above possible cases, we see that the distortion error in Subcase~2 is the smallest. Thus, the optimal set of two-means is  $\set{\frac{11}{16}, \frac{25}{16}}$ with quantization error $V_2=\frac{317}{3840}$, which is the proposition.
\end{proof}

\begin{prop1}\label{prop2}
The set $\set{0.400679, 1.202036, 1.734012}$ forms an optimal set of three-means with quantization error $V_3=0.0295695$.
\end{prop1}

\begin{proof}Let $\ga:=\set{a_1, a_2, a_3}$ be an optimal set of three-means with quantization error $V_3$.  Without any loss of generality, we can assume that  $0<a_1<a_2<a_3<2$. Consider the set of three elements $\gb:=\set{0.400679, 1.202036, 1.734012}$. The distortion error due to the set $\gb$ is given by
\[\int \min_{a\in \gb}(x-a)^2 dP=0.0295695.\]
Since $V_3$ is the quantization error for three-means, we have $V_3\leq 0.0295695$.
If $1\leq a_1$, then
\[V_3\geq \frac{1}{5} \int_0^1 (x-1)^2 \, dx=\frac{1}{15}=0.0666667>V_3,\]
which is a contradiction. Hence, $a_1<1$. Similarly, we can show that $a_3\leq 1$ leads to a contradiction. Hence, $1<a_3$, and recall Proposition~\ref{prop71} to  Proposition~\ref{prop73}. We now consider the following possible cases:

\tit{Case~1. $0<a_1<\frac{1}{2} (a_1+a_2)<a_2<\frac{1}{2}(a_2+a_3)\leq 1<a_3<2$. }

In this case the distortion error is given by
\begin{align*}
&V1(2, 1)=\int \min_{a\in \ga}(x-a)^2 dP=\frac{51 a_2^2 a_3-45 a_2 a_3^2-15 a_2^3-111 a_3^3+1920 a_3^2-4992 a_3+3712}{1920},
\end{align*}
the minimum value of which is $0.0708333$ and it occurs when $a_2=0.5$ and $a_3=1.5$.

\tit{Case~2. $0<a_1<\frac{1}{2} (a_1+a_2)<a_2\leq 1\leq \frac{1}{2} (a_2+a_3)<a_3<2$. }

In this case the distortion error is given by
\begin{align*}
&V2(2,1)=\int \min_{a\in \ga}(x-a)^2 dP\\
&=\frac{1}{480} (96 a_2^2 (a_3-3)-96 a_2 (a_3^2-3)-3 a_1^3+15 a_2 a_1^2-9 a_2^2 a_1+69 a_2^3-32 (3 a_3^3-24 a_3^2+48 a_3-29))
\end{align*}
the minimum value of which is $0.037037$ and it occurs when $a_1= 0.333333$, $a_2=1$, and $b_1=1.66667$.

\tit{Case~3. $0<a_1<\frac{1}{2} (a_1+a_2)\leq 1<a_2<\frac{1}{2} (a_2+a_3)<a_3<2$. }

In this case the distortion error is given by
\begin{align*}
&V1(1,2)=\int \min_{a\in \ga}(x-a)^2 dP\\
&=\frac{1}{120} (6 a_1^2 a_2-6 a_1 a_2^2+6 a_1^3+21 a_2^3+3 a_3^3+12 a_3^2-48 a_3+a_2^2 (9 a_3-60)-3 a_2 (5 a_3^2-8 a_3-8)+40),
\end{align*}
the minimum value of which is $0.0295695$ and it occurs when $a_1=0.400679, a_2=1.202036$, and $a_3=1.734012$.

\tit{Case~4. $0<a_1<1\leq \frac{1}{2} \left(a_1+a_2\right)<a_2<\frac{1}{2} \left(a_2+a_3\right)<a_3<2$. }

In this case the distortion error is given by
\begin{align*}
&V2(1,2)=\int \min_{a\in \ga}(x-a)^2 dP\\
&=\frac{1}{480} (3 a_1^2 (15 a_2-92)+a_1 (-51 a_2^2+24 a_2+240)+111 a_1^3+15 a_2^3+12 a_2^2-48 a_2-32),
\end{align*}
the minimum value of which is $0.0333333$ and it occurs when $a_1=0.5, a_2=1.5$.

Thus, considering all the possible cases, we see that $\set{0.400679, 1.202036, 1.734012}$ forms an optimal set of three-means with quantization error $V_3=0.0295695$, which is the proposition.
\end{proof}

\begin{lemma1} \label{lemma23}
Let $\ga_n$ be an optimal set of $n$-means for $P$ for any $n\geq 2$. Then, $\ga_n$ must contain an element from $[0, 1]$ and an element from $(1, 2]$, i.e., $\ga_n\ii [0, 1]\neq \es$ and $\ga_n\ii (1, 2]\neq \es$.
\end{lemma1}

\begin{proof} By Proposition~\ref{prop1} and Proposition~\ref{prop2}, the lemma is true for $n=2$ and $n=3$. We now prove the lemma for $n\geq 4$. Let $\ga_n:=\set{a_1, a_2, \cdots, a_n}$ be an optimal set of $n$-means with $0<a_1<a_2<\cdots<a_n$ where $n\geq 4$. Since $V_n$ is the quantization error for $n$-means with $n\geq 4$, we have $V_n<V_3$, i.e, $V_n<0.0295695$. We prove the lemma by contradiction. Suppose that $1\leq a_1$, then
\[V_n\geq \int_0^1(x-1)^2 dP=\frac 1{15}>V_3,\]
which leads to a contradiction. Next, suppose that $a_n\leq 1$, then
\[V_n\geq \int_1^2 (x-1)^2dP=\frac 4{15}>V_n,\]
which gives a contradiction. Thus, we can deduce that $a_1<1$ and $1<a_n$, i.e., the lemma is also true for $n\geq 4$. Thus, the proof of the lemma is complete.
\end{proof}

\begin{lemma1} \label{lemma24}
For $n\geq 2$, let $\ga_n$ be an optimal set of $n$-means for $P$. Assume that $\te{card}(\ga_n\ii [0, 1])=k$ and $\te{card}(\ga_n\ii (1, 2])=m$ for some positive integers $k$ and $m$. Then, either $\te{card}(\ga_{n+1}\ii [0, 1])=k+1$ and $\te{card}(\ga_{n+1}\ii (1, 2])=m$, or $\te{card}(\ga_{n+1}\ii [0, 1])=k$ and $\te{card}(\ga_{n+1}\ii (1, 2])=m+1$.
\end{lemma1}

\begin{proof}
Assume that $\te{card}(\ga_n\ii [0, 1])=k$ and $\te{card}(\ga_n\ii (1, 2])=m$ for some  positive integers $k$ and $m$. Notice that $k$ and $m$ depend on $n$, i.e., we can write $k:=k(n)$ and $m:=m(n)$. Let $V(k(n), m(n))$ be the corresponding distortion error.
By Proposition~\ref{prop1} and Proposition~\ref{prop2}, we know that $\ga_2=\set{\frac{11}{16}, \frac{25}{16}}$ and $\ga_3=\set{0.400679, 1.202036, 1.734012}$. Notice that here $\te{card}(\ga_2\ii [0, 1])=\te{card}(\ga_3\ii [0, 1])=1$, on the other hand, $\te{card}(\ga_2\ii (1, 2])=1$ and $\te{card}(\ga_3\ii (1, 2])=2$.
Thus, the lemma is true for $n=2$.

Let the lemma be true for $n=N$ for some given positive integer $N\geq 2$. Then, $\te{card}(\ga_N\ii [0, 1])=k(N) \te{ and } \te{card}(\ga_N\ii (1, 2])=m(N)$ imply that
either $\te{card}(\ga_{N+1}\ii [0, 1])=k(N)+1$ and $\te{card}(\ga_{N+1}\ii (1, 2])=m(N)$, or $\te{card}(\ga_{N+1}\ii [0, 1])=k(N)$ and $\te{card}(\ga_{N+1}\ii (1, 2])=m(N)+1$. Suppose that
$\te{card}(\ga_{N+1}\ii [0, 1])=k(N)+1$ and $\te{card}(\ga_{N+1}\ii (1, 2])=m(N)$ hold. Now, for the given $N$, by calculating the distortion errors $V(\ell, N+2-\ell)$ for all $1\leq \ell \leq N+1$, we see that the distortion error is smallest if $\te{card}(\ga_{N+2}\ii [0, 1])=k(N)+2 \te{ and } \te{card}(\ga_{N+2}\ii (1, 2]) =m(N), \te{ or if } \te{card}(\ga_{N+2}\ii [0, 1])=k(N)+1 \te{ and } \te{card}(\ga_{N+2}\ii (1, 2])=m(N)+1$, i.e., the lemma is true for $n=N+1$ whenever it is true for $n=N$. Similarly, we can show that the lemma is true for $n=N+1$ if $\te{card}(\ga_{N+1}\ii [0, 1])=k(N) \te{ and } \te{card}(\ga_{N+1}\ii (1, 2])=m(N)+1$ hold.  Thus, by the induction principle, the proof of the lemma is complete.
 \end{proof}
 \begin{remark1} \label{rem1}
 By Lemma~\ref{lemma23} and Lemma~\ref{lemma24}, we see that for $n\geq 2$, if $\ga_n$ is an optimal set of $n$-means for $P$, then there exist two positive integers $k:=k(n)$ and $m:=m(n)$ such that $\te{card}(\ga_n\ii [0, 1])=k$ and $\te{card}(\ga_n\ii (1, 2])=m$, and $k+m=n$. Let $F(k, m)$ be the corresponding quantization error for $n$-means. Notice that the function $F(k, m)$ satisfies
 \[F(k, m)=\min\set{V1(k,m), V2(k, m)}.\]
 \end{remark1}

 Let us now give the following definition.

\begin{defi1} \label{difi21}
Define the sequence $\set{a(n)}$ such that $a(n)=\lfloor \frac{8 (n+1)}{21}\rfloor$ for $n\in \D N$, i.e.,
\begin{align*}
 \set{a(n)}_{n=1}^\infty=&\set{0,1,1,1,2,2,3,3,3,4,4,4,5,5,6,6,6,7,7,8,8,8,9,9,9,10,10,11,11,11,12,\\
 &\qquad 12,12,13,13,14,14,14,15,15,16,16,16,17,17,17, \cdots},
\end{align*}
where $\lfloor x\rfloor$ represents the greatest integer not exceeding $x$.
\end{defi1}

The following algorithm helps us to determine the exact value of $k$, and so the value of $m:=n-k$, as mentioned in Remark~\ref{rem1}.
\subsubsection{\tbf{Algorithm.}}  For $n\geq 2$ let $k$ and $m:=n-k$ be the positive integers as defined in Remark~\ref{rem1}, and let $F(k, n-k)$ be the corresponding distortion error. Let $\set{a(n)}$ be the sequence  defined by Definition~\ref{difi21}. Then, the algorithm runs as follows:

$(i)$ Write $k:=a(n)$ and calculate $F(k, n-k)$.

$(ii)$ If $F(k-1, n-k+1)<F(k, n-k)$ replace $k$ by $k-1$ and return, else step $(iii)$.

$(iii)$ If $F(k+1, n-k-1)<F(k, n-k)$ replace $k$ by $k+1$ and return, else step $(iv)$.

$(iv)$ End.

When the algorithm ends, then the value of $k$, obtained, is the exact value of $k$ that $\ga_n$ contains from the closed interval $[0, 1]$.

\subsubsection{\tbf{Optimal sets of $n$-means and the $n$th quantization errors for all $n\geq 2$ when $p=\frac 15$.}} If $n=30$, then $a(n)=11$, and by the algorithm, we also obtain $k=11$; if $n=51$, then $a(n)=19$, and by the algorithm, we also obtain $k=19$. If $n=1000$, then  $a(n)=381$, and by the algorithm, we obtain $k=386$. Thus, we see that with the help of the sequence and the algorithm, we can easily determine the exact value of $k$ and $m:=n-k$ for any positive integer $n\geq 2$. Notice that $F(k, m)=\min\set{V1(k, m), V2(k, m)}$, i.e., either $V_n=V1(k, m)$, or $V_n=V2(k, m)$. Thus, in either case, once $k$ and $m$ are known, by using Proposition~\ref{prop71} to Proposition~\ref{prop73}, we can calculate the optimal sets $\ga_n$ of $n$-means and the corresponding quantization errors $V_n:=F(k, m)$ for any $n\geq 2$.

\subsection{Quantization for the mixed distribution $P$ when $p=\frac 1{3}$} \label{sec5}

Notice that if $p=\frac 1{3}$, then $E(X)=\frac{7}{6}$ and $V(X)=\frac{11}{36}$, i.e., the optimal set of one-mean for $p=\frac 13$ is $\set{\frac{7}{6}}$ and the corresponding quantization error is $V(X)=\frac{11}{36}$.

The optimal set of two- and three-means for the mixed distribution $P:=pP_1+(1-p)P_2$ when $p=\frac 13$ are given by the following two propositions.
\begin{prop1}   \label{prop11}
The optimal set of two-means is  $\set{\frac 12, \frac 32}$ with quantization error $V_2=\frac{1}{12}$.
\end{prop1}
\begin{proof}
Let $\ga:=\set{a_1, a_2}$ be an optimal set of two-means, where $0<a_1<a_2<2$. If $P=\frac 12 P_1+\frac 12 P_2$, i.e., if $P$ is a uniform distribution, then by Proposition~\ref{prop31}, we know that the optimal set of two-means is given by $\set{\frac 12, \frac 32}$. Notice that here the mixed distribution is given by $P=\frac 13P_1+\frac 23P_2$, i.e., we are giving more weight to the right-hand interval $J_2$. This leads us to conclude that the Voronoi region of $a_1$ may contain some elements from $J_2$ yielding the fact that $\frac 12\leq a_1$, i.e, the boundary $\frac 12(a_1+a_2)$ of the two Voronoi regions must satisfy $1\leq \frac 12(a_1+a_2)$. Under this condition, we have the distortion error for the optimal set $\ga=\set{a_1, a_2}$ as
\begin{align*}
& \int\min_{a\in \ga}(x-a)^2 dP=\frac{1}{3} \int_0^1 (x-a_1){}^2 \, dx +\frac{2}{3} \int_1^{\frac{1}{2} (a_1+a_2)} (x-a_1){}^2 \, dx+\frac{2}{3} \int_{\frac{1}{2} (a_1+a_2)}^2 (x-a_2){}^2 \, dx\\
&=\frac{1}{6} \Big(a_1^2 (a_2-2)-a_1 (a_2^2-2)+a_1^3-a_2^3+8 a_2^2-16 a_2+10\Big),
\end{align*}
the minimum value of which is $\frac 1{12}$, and it occurs when $a_1=\frac 12$ and $a_2=\frac 32$. Thus, the proposition is yielded.
 \end{proof}

\begin{remark1}
It is interesting to see that the optimal set of two-means for the mixed distribution $P:=\frac 13 P_1+\frac 23 P_2$ is same as the optimal set of two-means for the uniform distribution $P$ given by $P:=\frac 12 P_1+\frac 12 P_2$.
\end{remark1}

\begin{prop1}\label{prop21}
The set $\set{0.380129, 1.14039, 1.71346}$ forms an optimal set of three-means with quantization error $V_3=0.0343006$.
\end{prop1}

\begin{proof}Let $\ga:=\set{a_1, a_2, a_3}$ be an optimal set of three-means with quantization error $V_3$.  Without any loss of generality, we can assume that  $0<a_1<a_2<a_3<2$. Consider the set of three elements $\gb:=\set{0.380129, 1.14039, 1.71346}$. The distortion error due to the set $\gb$ is given by
\[\int \min_{a\in \gb}(x-a)^2 dP=0.0343006.\]
Since $V_3$ is the quantization error for three-means, we have $V_3\leq 0.0343006$. As shown in the proof of Proposition~\ref{prop2}, we can show that $a_1<1$ and $1<a_3$. We now consider the following possible cases:

\tit{Case~1. $0<a_1<\frac{1}{2} (a_1+a_2)<a_2<\frac{1}{2}(a_2+a_3)\leq 1<a_3<2$. }

In this case the distortion error is given by
$V1(2, 1)= 0.0625000.$

\tit{Case~2. $0<a_1<\frac{1}{2} (a_1+a_2)<a_2\leq 1\leq \frac{1}{2} (a_2+a_3)<a_3<2$. }

In this case the distortion error is given by
$V2(2,1)= 0.037037$.

\tit{Case~3. $0<a_1<\frac{1}{2} (a_1+a_2)\leq 1<a_2<\frac{1}{2} (a_2+a_3)<a_3<2$. }

In this case the distortion error is given by
$ V1(1,2)=0.0343006$.

\tit{Case~4. $0<a_1<1\leq \frac{1}{2} \left(a_1+a_2\right)<a_2<\frac{1}{2} \left(a_2+a_3\right)<a_3<2$. }

In this case the distortion error is given by
$V2(1,2)=0.0416667.
$

Thus, considering all the possible cases we see that in Case~3 the distortion error is smallest, and it occurs when $a_1=0.380129$, $a_2=1.14039$, and $a_3=1.71346$. Thus, the proof of the proposition is complete.
\end{proof}

The following two lemmas, which are analogous to Lemma~\ref{lemma23} and Lemma~\ref{lemma24} are also true here.
\begin{lemma1} \label{lemma231}
Let $\ga_n$ be an optimal set of $n$-means for $P$ where $n\geq 2$. Then, $\ga_n$ must contain an element from $[0, 1]$ and an element from $(1, 2]$, i.e., $\ga_n\ii [0, 1]\neq \es$ and $\ga_n\ii (1, 2]\neq \es$.
\end{lemma1}

\begin{lemma1} \label{lemma241}
For $n\geq 2$, let $\ga_n$ be an optimal set of $n$-means for $P$. Assume that $\te{card}(\ga_n\ii [0, 1])=k$ and $\te{card}(\ga_n\ii (1, 2])=m$ for some positive integers $k$ and $m$. Then, either $\te{card}(\ga_{n+1}\ii [0, 1])=k+1$ and $\te{card}(\ga_{n+1}\ii (1, 2])=m$, or $\te{card}(\ga_{n+1}\ii [0, 1])=k$ and $\te{card}(\ga_{n+1}\ii (1, 2])=m+1$.
\end{lemma1}

 \begin{remark1} \label{rem11}
 By Lemma~\ref{lemma231} and Lemma~\ref{lemma241}, we see that for $n\geq 2$, if $\ga_n$ is an optimal set of $n$-means for $P$, then there exist two positive integers $k:=k(n)$ and $m:=m(n)$ such that $\te{card}(\ga_n\ii [0, 1])=k$ and $\te{card}(\ga_n\ii (1, 2])=m$, and $k+m=n$. Let $F(k, m)$ be the corresponding quantization error for $n$-means. Notice that the function $F(k, m)$ satisfies
 \[F(k, m)=\min\set{V1(k,m), V2(k, m)}.\]
 \end{remark1}

Let us now give the following sequence and the algorithm which are analogous to the sequence and the algorithm defined in the previous subsection.
\begin{defi1} \label{difi22}
Define the sequence $\set{b(n)}$ such that $b(n)=\lfloor \frac{4 (n+1)}{7}\rfloor$ for $n\in \D N$, i.e.,
\begin{align*}
 \set{b(n)}_{n=1}^\infty=&\set{1,1,2,2,3,4,4,5,5,6,6,7,8,8,9,9,10,10,11,12,12,13,13,14,14,\\
 &\qquad 15,16,16,17,17,18,18,19,20,20,21,21,22,22,23, \cdots},
\end{align*}
where $\lfloor x\rfloor$ represents the greatest integer not exceeding $x$.
\end{defi1}

\subsubsection{\tbf{Algorithm.}}  For $n\geq 2$ let $k:=n-m$ and $m$ be the positive integers as defined in Remark~\ref{rem11}, and let $F(n-m, m)$ be the corresponding distortion error. Let $\set{b(n)}$ be the sequence  defined by Definition~\ref{difi22}. Then, the algorithm runs as follows:

$(i)$ Write $m:=b(n)$ and calculate $F(n-m, m)$.

$(ii)$ If $F(n-m+1, m-1)<F(n-m, m)$ replace $m$ by $m-1$ and return, else step $(iii)$.

$(iii)$ If $F(n-m-1, m+1)<F(n-m, m)$ replace $m$ by $m+1$ and return, else step $(iv)$.

$(iv)$ End.

When the algorithm ends, then the value of $m$, obtained, is the exact value of $m$ that $\ga_n$ contains from the interval $(1, 2]$.

\subsubsection{\tbf{Optimal sets of $n$-means and the $n$th quantization errors for all $n\geq 2$ when $p=\frac 13$.}} If $n=21$, then $b(n)=12$, and by the algorithm we also obtain $m=12$; if $n=100$, then $b(n)=57$, and by the algorithm we obtain $m=56$. If $n=500$, then  $b(n)=286$, and by the algorithm, we obtain $m=279$. Thus, we see that with the help of the sequence and the algorithm, we can easily determine the exact values of $k:=n-m$ and $m$ for any positive integer $n\geq 2$. Once $k$ and $m$ are known, by using Proposition~\ref{prop71} to Proposition~\ref{prop73}, we can calculate the optimal sets $\ga_n$ of $n$-means and the corresponding quantization errors $V_n:=F(k, m)$ for any positive integer $n\geq 2$.

\section{Mixed distribution $P$ with support the union of the disconnected line segments $[0, \frac 13]$ and $[\frac 23, 1]$}
\subsection{Basic preliminaries} \label{sec33}
 Let $P_1$ and $P_2$ be two uniform distributions, respectively, defined on the intervals given by
 \[J_1:=[0, \frac 13], \te{ and } J_2:=[\frac 23, 1].\]
 Let $f_1$ and $f_2$ be their respective density functions. Then,
\[f_1(x)=\left \{\begin{array} {cc}
3 & \te{ if } x\in [0, \frac 13], \\
0 & \te{otherwise};
\end{array}
\right. \te{ and } f_2(x)=\left \{\begin{array} {cc}
3 & \te{ if } x\in [\frac 23, 1], \\
0 & \te{otherwise}.
\end{array}
\right.\]
The underlying mixed distribution considered is given by $P:=pP_1+(1-p)P_2$, where $0<p<1$.  By $E(X)$, it is meant the expectation of a random variable $X$ with probability distribution $P$, and $V(X)$ represents the variance of $X$.

 \begin{prop}
Let $P$ be the mixed distribution defined by $P=p P_1+(1-p) P_2$. Then, $E(X)=\frac{1}{6} (5-4 p)$, and $V(X)=\frac{1}{108}  (-48 p^2+48 p+1 )$.
\end{prop}

\begin{proof}
We have
\[E(X)=\int x dP=p \int x d(P_1(x))+(1-p)\int x d(P_2(x))=p \int_{J_1}3x \, dx+(1-p) \int_{J_2} 3 x \, dx\]
yielding $E(X)=\frac{1}{6} (5-4 p)$, and
\[V(X)=\int (x-E(X))^2 dP=p \int(x-E(X))^2 d(P_1(x))+(1-p) \int(x-E(X))^2 d(P_2(x)),\]
implying
$V(X)=\frac{1}{108}  (-48 p^2+48 p+1 )$, and thus, the proposition is yielded.
\end{proof}

\begin{remark}
The optimal set of one-mean is the set $\set{\frac{1}{6} (5-4 p)}$, and the corresponding quantization error is the variance $V:=V(X)$ of a random variable with distribution $P:=pP_1+(1-p)P_2$.
\end{remark}
In the next two subsections, taking $p=\frac 1{100}$,  $p=\frac 25$, and $p=\frac 1{1000}$, for all positive integers $n$, we calculate the optimal set of $n$-means and the $n$th quantization errors for the mixed distribution $P:=pP_1+(1-p)P_2$ with support the union of two disconnected line segments $[0, \frac 13]$ and $[\frac 23, 1]$.
\subsection{Quantization for the mixed distribution $P$ when $p=\frac 1{100}$} \label{subsec44}

Notice that if $p=\frac 1{100}$, then $E(X)=\frac{62}{75}$ and $V(X)=\frac{461}{33750}$, i.e., the optimal set of one-mean for $p=\frac 15$ is $\set{\frac{62}{75}}$ and the corresponding quantization error is $V(X)=\frac{461}{33750}$.

\begin{prop1} \label{prop0000}
The optimal set of two-means is  $\set{0.731517, 0.910506}$ with quantization error $V_2=0.005682$.
\end{prop1}
\begin{proof}
Let $\ga:=\set{a_1, a_2}$ be an optimal set of two-means. Since the elements in an optimal set are the conditional expectations in their own Voronoi regions, without any loss of generality, we can assume that $0<a_1<a_2<1$. If $\frac 13<a_1<a_2<\frac 23$, then the quantization error can be strictly reduced by moving the element $a_1$ to $\frac 13$, and $a_2$ to $\frac 23$. Thus, $\frac 13<a_1<a_2<\frac 23$ is not possible. Let us now discuss all the possible cases:

\tit{Case~1. $0<a_1<a_2\leq \frac 13$. }

Since the boundary of the Voronoi region is $\frac 12(a_1+a_2)$, we have the distortion error as
\begin{align*}
&\int \min_{a\in \ga}(x-a)^2 dP=\int_{0}^{\frac {a_1+a_2}{2}}(x-a_1)^2 dP+\int_{\frac{a_1+a_2}{2}}^{\frac 13} (x-a_2)^2 dP+\int_{\frac 23}^1(x-a_2)^2 dP\\
&=\frac{1}{100} \int_0^{\frac{1}{2} \left(a_1+a_2\right)} 3 \left(x-a_1\right){}^2 \, dx+\frac{1}{100} \int_{\frac{1}{2} \left(a_1+a_2\right)}^{\frac{1}{3}} 3 \left(x-a_2\right){}^2 \, dx+\frac{99}{100} \int_{\frac{2}{3}}^1 3 \left(x-a_2\right){}^2 \, dx\\
&=\frac{81 a_1^3+81 a_2 a_1^2-81 a_2^2 a_1-81 a_2^3+10800 a_2^2-17856 a_2+7528}{10800},
\end{align*}
the minimum value of which is $\frac{3119}{12150}$ and it occurs when $a_1=\frac 19, $ and $a_2=\frac 13$.

\tit{Case~2. $0<a_1\leq \frac 13 <a_2< \frac 23$. }

In this case, the boundary $\frac 12(a_1+a_2)$ of the Voronoi regions of $a_1$ and $a_2$ must satisfy $0<a_1<\frac 12(a_1+a_2)<\frac 13$, otherwise the quantization error can be strictly reduced by moving the element $a_2$ to $\frac 23$. Hence, the distortion error in this case is given by
\begin{align*}
&\int \min_{a\in \ga}(x-a)^2 dP\\
&=\frac{1}{100} \int_0^{\frac{1}{2} \left(a_1+a_2\right)} 3 \left(x-a_1\right){}^2 \, dx+\frac{1}{100} \int_{\frac{1}{2} \left(a_1+a_2\right)}^{\frac{1}{3}} 3 \left(x-a_2\right){}^2 \, dx+\frac{99}{100} \int_{\frac{2}{3}}^1 3 \left(x-a_2\right){}^2 \, dx\\
&=\frac{81 a_1^3+81 a_2 a_1^2-81 a_2^2 a_1-81 a_2^3+10800 a_2^2-17856 a_2+7528}{10800},
\end{align*}
the minimum value of which is $\frac{89}{2430}$ and it occurs when $a_1=\frac 29, $ and $a_2=\frac 23$.

\tit{Case~3. $0<a_1\leq \frac 13 <\frac 23\leq a_2$. }

In this case, the Voronoi region of $a_1$ does not contain any element from $J_2$, if it does, then we must have $\frac 12(a_1+a_2)>\frac 23$ implying $a_2>\frac 43-a_1\geq \frac 43-\frac 13=1$, which is a contradiction as $a_2<1$. Similarly, we can show that the Voronoi region of $a_2$ does not contain any element from $J_1$. This yields the fact that
\[a_1=E(X : X \in J_1)=\frac 16, \te{ and } a_2=E(X : X \in J_2)=\frac 56,\]
with distortion error
\[\int \min_{a\in \ga}(x-a)^2 dP=\frac{1}{100} \int_0^{\frac{1}{3}} 3 (x-\frac{1}{6})^2 \, dx+\frac{99}{100} \int_{\frac{2}{3}}^1 3(x-\frac{5}{6})^2 \, dx=\frac{1}{108}.\]

\tit{Case~4. $\frac 13<a_1 \leq \frac 23< a_2$. }

In this case, the boundary $\frac 12(a_1+a_2)$ of the Voronoi regions of $a_1$ and $a_2$ must satisfy $\frac 23<\frac 12(a_1+a_2)<a_2<1$, otherwise the quantization error can be strictly reduced by moving the element $a_1$ to $\frac 13$. Hence, the distortion error in this case is given by
\begin{align*}
&\int \min_{a\in \ga}(x-a)^2 dP\\
&=\frac{1}{100} \int_0^{\frac{1}{3}} 3 \left(x-a_1\right){}^2 \, dx+\frac{99}{100} \int_{\frac{2}{3}}^{\frac{1}{2} \left(a_1+a_2\right)} 3 \left(x-a_1\right){}^2 \, dx+\frac{99}{100} \int_{\frac{1}{2} \left(a_1+a_2\right)}^1 3 \left(x-a_2\right){}^2 \, dx\\
&=\frac{8019 a_1^3+27 \left(297 a_2-788\right) a_1^2-9 \left(891 a_2^2-1580\right) a_1-8019 a_2^3+32076 a_2^2-32076 a_2+7528}{10800},
\end{align*}
the minimum value of which is $\frac{1}{150}$ and it occurs when $a_1=\frac 23, $ and $a_2=\frac 89$.

\tit{Case~5. $\frac 23<a_1  < a_2<1$. }

In this case, the distortion error is given by
\begin{align*}
&\int \min_{a\in \ga}(x-a)^2 dP\\
&=\frac{1}{100} \int_0^{\frac{1}{3}} 3 \left(x-a_1\right){}^2 \, dx+\frac{99}{100} \int_{\frac{2}{3}}^{\frac{1}{2} \left(a_1+a_2\right)} 3 \left(x-a_1\right){}^2 \, dx+\frac{99}{100} \int_{\frac{1}{2} \left(a_1+a_2\right)}^1 3 \left(x-a_2\right){}^2 \, dx\\
&=\frac{8019 a_1^3+27 \left(297 a_2-788\right) a_1^2-9 \left(891 a_2^2-1580\right) a_1-8019 a_2^3+32076 a_2^2-32076 a_2+7528}{10800},
\end{align*}
the minimum value of which is $0.005682$ and it occurs when $a_1=0.731517,$ and $a_2= 0.910506$.

Comparing the distortion errors obtained in all the above possible cases, we see that the distortion error in Case~5 is smallest. Thus, the optimal set of two-means is  $\set{0.731517, 0.910506}$ with quantization error $V_2=0.005682$, which is the proposition.
\end{proof}

\begin{prop1} \label{prop0101}
Optimal set of three-means is $\set{\frac{1}{6},\frac{3}{4},\frac{11}{12}}$ with quantization error $V_3= \frac{103}{43200}$.
\end{prop1}

\begin{proof}
Let $\ga=\set{a_1, a_2, a_3}$ be an optimal set of three-means. Proposition~\ref{prop0} implies that if $\ga$ contains an element from the open interval $(\frac 13, \frac 23)$, it cannot contain more than one element from the open interval $(\frac 13, \frac 23)$. First, we assume that $\ga$ contains one element from $J_1$, and two elements from $J_2$. Then, $0<a_1\leq \frac 13<\frac 23\leq a_2<a_3<1$ yielding the fact that the Voronoi region of $a_1$ does not contain any element from $J_1$, and the Voronoi region of $a_2$, and so of $a_3$ cannot not contain any element from $J_1$. This yields $a_1=\frac 16$, $a_2=\frac 34$, and $a_3=\frac {11}{12}$ with distortion error
\begin{align*} \int \min_{a\in \ga}(x-a)^2 dP=\frac{1}{100} \int_0^{\frac{1}{3}} 3 (x-\frac{1}{6})^2 \, dx+\frac{99}{100} \int_{\frac{2}{3}}^{\frac{5}{6}} 3 (x-\frac{3}{4})^2 \, dx+\frac{99}{100} \int_{\frac{5}{6}}^1 3 (x-\frac{11}{12})^2 \, dx
\end{align*}
yielding \[\int \min_{a\in \ga}(x-a)^2 dP=\frac{103}{43200}.\]
Since $V_3$ is the quantization error for three-means, we have $V_3\leq \frac{103}{43200}=0.00238426$. If $a_3<\frac 23$, then
\[V_3\geq \frac{99}{100} \int_{\frac{2}{3}}^1 3 (x-\frac{2}{3} )^2 \, dx=\frac{11}{300}>V_3,\]
which is a contradiction. Hence, we can assume that $\frac 23<a_3$. Suppose that $a_2<\frac 23$. Then,
\[V_3\geq \int_{J_3}\min_{a\in \ga}(x-a)^2 dP=\int_{J_3}\min_{a\in \set{a_2, a_3}}(x-a)^2 dP\geq \int_{J_3}\min_{a\in \set{\frac 23, a_3}}(x-a)^2 dP\]
implying
\begin{align*} V_3&\geq \frac{99}{100} \int_{\frac{2}{3}}^{\frac{1}{2}  (a_3+\frac{2}{3})} 3 (x-\frac{2}{3})^2 \, dx+\frac{99}{100} \int_{\frac{1}{2} \left(a_3+\frac{2}{3}\right)}^1 3  (x-a_3 ){}^2 \, dx\\
&=-\frac{11  (81 a_3^3-270 a_3^2+288 a_3-100 )}{1200},
\end{align*}
the minimum value of which is $\frac{11}{2700}$ and it occurs when $a_3=\frac 89$. Notice that $\frac{11}{2700}=0.00407407>V_3$, which leads to a contradiction. Hence, we can assume that $\frac 23\leq a_2<a_3<1$. Notice that for $\frac 23\leq a_2<a_3<1$, the Voronoi region of $a_2$ does not contain any element from $J_1$. Suppose that $\frac 23\leq a_1$. Then,
\[V_3>\frac{1}{100} \int_0^{\frac{1}{3}} 3 (x-\frac{2}{3})^2 \, dx=\frac{7}{2700}=0.00259259>V_3,\]
which leads to a contradiction. So, we can assume that $a_1<\frac 23$. Suppose that $\frac 13<a_1<\frac 23$.
Then, we must have $\frac 23<\frac 12(a_1+a_2)<a_2<a_3<1$ yielding the distortion error as
\begin{align*} &\int \min_{a\in \ga}(x-a)^2 dP\\
&=\frac{1}{100} \int_0^{\frac{1}{3}} 3 \left(x-a_1\right){}^2 \, dx+\frac{99}{100} \int_{\frac{2}{3}}^{\frac{1}{2} \left(a_1+a_2\right)} 3 \left(x-a_1\right){}^2 \, dx+\frac{99}{100} \int_{\frac{1}{2} \left(a_1+a_2\right)}^{\frac{1}{2} \left(a_2+a_3\right)} 3 \left(x-a_2\right){}^2 \, dx\\
&\qquad \qquad \qquad \qquad +\frac{99}{100} \int_{\frac{1}{2} \left(a_2+a_3\right)}^1 3 \left(x-a_3\right){}^2 \, dx\\
&=\frac 1{10800}\Big(8019 a_1^3+27 (297 a_2-788) a_1^2-9 (891 a_2^2-1580) a_1-8019 a_3^3-8019 (a_2-4) a_3^2\\
&\qquad \qquad \qquad \qquad +8019 (a_2^2-4) a_3+7528\Big)
\end{align*}
the minimum value of which is $\frac{137}{33750}$, and it occurs when $a_1=\frac 23$, $a_2=\frac 45$, and $a_3=\frac{14}{15}$. Notice that $\frac{137}{33750}=0.00405926>V_3$, and thus a contradiction arises. Hence, we can assume that $a_1\leq \frac 13$. Thus, as described before, we deduce that $a_1=\frac 16$, $a_2=\frac 34$, and $a_3=\frac {11}{12}$, and the quantization error for three-means is $V_3= \frac{103}{43200}$. Thus, the proof of the proposition is complete.
\end{proof}

\begin{lemma1} \label{lemma1111}
For $n\geq 3$ , let $\ga_n$ be an optimal set of $n$-means for $P$. Then, $\ga_n\ii J_1\neq \es$ and $\ga_n\ii J_2\neq \es$.
\end{lemma1}
\begin{proof} By Proposition~\ref{prop0101}, the lemma is true for $n=3$. Let us now prove the lemma for $n\geq4$. The distortion error due to the set $\gb:=\set{\frac 16, \frac 23+\frac 1{18}, \frac 23+\frac 3{18},\frac 23+\frac 5{18}}=\set{\frac 16, \frac {13}{18}, \frac{15}{18}, \frac{17}{18}}$ is given by
\[\int\min_{a\in \gb} (x-a)^2 dP=\int_{J_1}(x-\frac 16)^2 dP+\int_{J_2}\min_{a\in (\gb\setminus\set {\frac 16})}(x-a)^2 dP=\frac{1}{900}.\]
Since $V_n$ is the quantization error for $n$-means, where $n\geq 4$, we have $V_n\leq \frac{1}{900}=0.00111111$. For $n\geq 4$,
let $\ga_n:=\set{a_1, a_2, \cdots, a_n}$ be an optimal set of $n$-means for $P$ such that $0<a_1<a_2<\cdots<a_n<1$. If $a_n< \frac 23$, then
\[V_n>\int_{J_2}(x-\frac 23)^2 dP=\frac{11}{300}=0.0366667>V_n,\]
which leads to a contradiction. Hence, we can assume that $a_n\geq \frac 23$, i.e., $\ga_n\ii J_2\neq \es$. We now show that $\ga_n\ii J_1\neq \es$. If $\frac 12\leq a_1$, then
\[V_n>\int_{J_1}(x-\frac 12)^2 dP=\frac{13}{10800}=0.0012037>V_n,\]
which yields a contradiction. Hence, we can assume that $a_1<\frac 12$. Assume that $\frac 13\leq a_1< \frac 12$. Then, by Proposition~\ref{prop0}, we must have $\frac 12(a_1+a_2)>\frac 23$ yielding $a_2>\frac 43-a_1\geq \frac 43-\frac 12=\frac 56$, and so,
\[V_n>\int_{J_1}(x-\frac 13)^2 dP+\int_{[\frac 23, \frac 56]}(x-\frac 12)^2 dP=\frac{2681}{21600}=0.12412>V_n,\]
which leads to a contradiction. Hence, we can assume that $a_1<\frac 13$, i.e., $\ga_n\ii J_1\neq \es$. Thus, the proof of the lemma is complete.
\end{proof}

\begin{lemma1} \label{lemma1212}
An optimal set of four-means does not contain any element from the open interval $(\frac 13, \frac 23)$.
\end{lemma1}
\begin{proof}
Let $\ga:=\set{a_1, a_2, a_3, a_4}$, where $0<a_1<a_2<a_3<a_4<1$, be an optimal set of four-means. As mentioned in the proof of Lemma~\ref{lemma1111}, we have $V_4\leq \frac{1}{900}=0.00111111$.
Suppose that $a_2\leq \frac 23$. Then,
\begin{align*}
V_4&>\frac{99}{100} \int_{\frac{2}{3}}^{\frac{1}{2} \left(a_3+\frac{2}{3}\right)} 3 (x-\frac{2}{3})^2 \, dx+\frac{99}{100} \int_{\frac{1}{2} \left(a_3+\frac{2}{3}\right)}^{\frac{1}{2} \left(a_3+a_4\right)} 3 \left(x-a_3\right){}^2 \, dx+\frac{99}{100} \int_{\frac{1}{2} \left(a_3+a_4\right)}^1 3 \left(x-a_4\right){}^2 \, dx\\
&=\frac{11 \left(-81 a_4^3+324 a_4^2-324 a_4+27 a_3^2 \left(3 a_4-2\right)-9 a_3 \left(9 a_4^2-4\right)+100\right)}{1200}
\end{align*}
the minimum value of which is $\frac{11}{7500}$, and it occurs when $a_3=\frac{4}{5}$ and $,a_4=\frac{14}{15}$ implying
\[V_4>\frac{11}{7500}=0.00146667>V_4,\]
which leads to a contradiction. Thus, we can assume that $\frac 23<a_2$, and so $\frac 23<a_2<a_3<a_4<1$. Again, by Lemma~\ref{lemma1111}, we see that $a_1<\frac 13$. Hence, an optimal set of four-means does not contain any element from the open interval $(\frac 13, \frac 23)$, which is the lemma.
\end{proof}

\begin{remark1} \label{rem00}
Proceeding in the similar way as Lemma~\ref{lemma1212}, we can show that the optimal set of five-means does not contain any element from the open interval $(\frac 13, \frac 23)$.
\end{remark1}

\begin{prop1}\label{prop2121}
For $n\geq 3$ , let $\ga_n$ be an optimal set of $n$-means for $P$. Then, $\ga_n$ does not contain any element from the open interval $(\frac 13, \frac 23)$. Moreover, the Voronoi region of any element in $\ga_n\ii J_1$ does not contain any element from $J_2$, and the Voronoi region of any element in $\ga_n\ii J_2$ does not contain any element from $J_1$.
\end{prop1}
\begin{proof} By Proposition~\ref{prop0101}, Lemma~\ref{lemma1212} and Remark~\ref{rem00}, the proposition is true for $n=3, 4, 5$. We now prove the proposition for $n\geq 6$. Let $\ga_n:=\set{a_1, a_2, \cdots, a_n}$, where $n\geq 6$, be an optimal set of $n$-means. Without any loss of generality, we can assume that $0<a_1<a_2<\cdots<a_n<1$. Let us now consider the set of six elements $\gb:=\set{\frac{1}{6},\frac{7}{10},\frac{23}{30},\frac{5}{6},\frac{9}{10},\frac{29}{30}}$. By routine calculation, the distortion error due to the set $\gb$ is given by
\[\int\min_{a\in \gb}(x-a)^2 dP=\frac 1{100} \int_{J_1}(x-\frac 16)^2 dP_1+\frac{99}{100}\int_{J_2}\min_{b\in (\gb \setminus{\set{\frac 16}})}(x-b)^2 dP_2=\frac{31}{67500},\]
and so, $V_6\leq \frac{31}{67500}=0.000459259$.
Since $V_n$ is the quantization error for six-means with $n\geq 6$, we have $V_n\leq V_6\leq 0.000459259$. By Lemma~\ref{lemma1111}, we know that $a_1<\frac 13$ and $a_n>\frac 23$. Let $k$ be the largest positive integer such that $a_k\leq \frac 13$. For the sake of contradiction, assume that $\ga_n$ contains an element from the open interval $(\frac 13, \frac 23)$. Then, by Proposition~\ref{prop0}, we must have $a_{k+1}\in (\frac 13, \frac 23)$, and $\frac 23\leq a_{k+2}$. The following two cases can arise:

\tit{Case~1.} $\frac 13<a_{k+1}\leq \frac 12$.

Then, the Voronoi region of $a_{k+1}$ must contain elements from $J_2$, i.e., $\frac 12(a_{k+1}+a_{k+2})\geq \frac 23$ implying $a_{k+2}\geq \frac 43-a_{k+1}\geq \frac 43-\frac 12=\frac 56$, otherwise the quantization error can be strictly reduced by moving the element $a_{k+1}$ to $\frac 13$.
Then,
\begin{align*}
V_n\geq \int_{[\frac 23, \frac 56]}(x-\frac 56)^2 dP=\frac{11}{2400}=0.00458333>V_n,
\end{align*}
which is a contradiction.

\tit{Case~2.} $\frac 12\leq a_{k+1}<\frac 23$.

Then, we must have $\frac 12(a_k+a_{k+1})\leq \frac 13$ implying $a_k\leq \frac 23-a_{k+1}\leq \frac 23-\frac 12=\frac 16$, and so
\begin{align*}
V_n\geq \int_{[\frac 16, \frac 13]}(x-\frac 16)^2 dP=\frac{1}{21600}=0.0000462963>V_n,
\end{align*}
which leads to a contradiction.

By Case~1 and Case~2, we deduce that $\ga_n$ does not contain any element from the open interval $(\frac 13, \frac 23)$. Thus, $\frac 23\leq a_{k+1}$. To complete the proof, assume that the Voronoi region of $a_k$ contains elements from $J_2$. Then, $\frac 12(a_k+a_{k+1})>\frac 23$ implying $a_{k+1}>\frac 43-a_k\geq \frac 43-\frac 13=1$, which is a contradiction. Similarly, we can show that if the Voronoi region of $a_{k+1}$ contains elements from $J_1$, then a contradiction arises. Thus, the proof of the proposition is complete.
\end{proof}

We are now ready to prove the following theorem.

\begin{theorem1} \label{th3232}
For $n\geq 3$ , let $\ga_n$ be an optimal set of $n$-means for $P$. Let $\te{card}(\ga_n\ii J_1)=k$. Then, $\ga_n$ contains $k$ elements from $J_1$, and $(n-k)$ elements from $J_2$, i.e., $\ga_n(P)=\ga_{k}(P_1)\uu \ga_{n-k}(P_2)$ with quantization error
\[V_n(P)=\frac{1}{324} \Big(\frac{1}{k^2}+\frac{2}{(n-k)^2}\Big).\]
\end{theorem1}
\begin{proof}
By Proposition~\ref{prop2121}, we have $\ga_n\ii J_1\neq \es$ and $\ga_n\ii J_2\neq \es$. Thus, there exist two positive integers $n_1$ and $n_2$ such that $\te{card}(\ga_n\ii J_1)=n_1$, and $\te{card}(\ga_n\ii J_2)=n_2$. Again, by  Proposition~\ref{prop2121}, $\ga_n$ does not contain any element from the open interval $(\frac 13, \frac 23)$, and so we have $n=n_1+n_2$. Hence, by taking $n_1=k$, we see that $\ga_n$ contains $k$ elements from $J_1$, and $(n-k)$ elements from $J_2$. Again, by Proposition~\ref{prop2121}, we know that the Voronoi region of any element in $\ga_n\ii J_1$ does not contain any element from $J_2$, and the Voronoi region of any element from $\ga_n\ii J_2$ does not contain any element from $J_1$. This implies the fact that $\ga_n(P)=\ga_{k}(P_1)\uu \ga_{n-k}(P_2)$, and the corresponding quantization error is given by
\[V_n(P)=\frac 1{100} V_{k}(P_1)+\frac {99}{100}V_{n-k}(P_2)=\frac{1}{10800} \Big(\frac{1}{k^2}+\frac{99}{(n-k)^2}\Big).\]
Thus, the proof of the theorem is complete.
\end{proof}

\begin{remark1} \label{rem5656}
Let $k$ be the positive integer as stated in Theorem~\ref{th3232}. Then,  $\ga_k(P_1)$ and $\ga_{n-k}(P_2)$ are known by Theorem~\ref{th3232}. Thus, once $k$ is known, we can easily determine the optimal sets of $n$-means and the $n$th quantization errors for all $n\in \D N$ with $n\geq 3$. For $n\geq 3$, consider the real valued function
\[F(n, x)=\frac{1}{10800} \Big(\frac{1}{x^2}+\frac{99}{(n-x)^2}\Big)\]
defined in the domain $1\leq x\leq n-1$. Notice that $F(n, x)$ is concave upward, and so $F(n, x)$ attains its minimum at a unique $x$ in the interval $[1, n-1]$. Thus, we can say that for a given positive integer $n\geq 3$, there exists a unique positive integer $k$ for which $F(n, k)$ is minimum if $x$ ranges over the positive integers in the interval $[1, n-1]$.
\end{remark1}

\begin{remark1} \label{rem5757}
For any positive integer $n\geq 3$, let us write \[V(j, n-j):=\frac{1}{100} V_j(P_1)+\frac{99}{100} V_{n-j}(P_2),\]
where $1\leq j\leq n-1$. For a given $n$ let $k:=k(n)$ be the positive integer as stated in Theorem~\ref{th3232}. Then, we have $V_n=V(k, n-k)$. Notice that
\[V_n=V(k, n-k)=\min\set{V(j, n-j) : 1\leq j\leq n-1}.\]
Moreover, if we order the elements of the set $\set{V(j, n-j) : 1\leq j\leq n-1}$ in a sequence as
\[\set{V(1, n-1), V(2, n-2), \cdots, V(n-1, 1)}\]
then $V (k, n-k)$ is the $k$th term in the sequence. Using this fact, for a given $n$ we can easily determine the value of the positive integer $k$ as follows:

Define the function
\begin{equation} \label{eq000000} f : \D N \to \D N \te{ such that } f(n)=k,
\end{equation}  where $k$ is the unique positive integer such that
\begin{equation*} \label{eq0}  V_n:=V(k, n-k)=\min\set{V(j, n-j) : j\in \D N, \, 1\leq j\leq n-1}.
\end{equation*}
For a given positive integer $n$, once $k:=k(n)$ is known, using Theorem~\ref{th3232}, we can determine the optimal set of $n$-means and the corresponding quantization error.
\end{remark1}
In the following example, we calculate the values of $k$ for different values of $n$. For such calculations, we have used Mathematica.

\begin{exam} \label{exam1} Recall the function $f$ defined by \eqref{eq000000}. Then, we see that
\begin{align*} \set{f(n)}_{n=3}^\infty&=\{1,1,1,1,1,2,2,2,2,2,2,3,3,3,3,3,3,4,4,4,4,4,4,5,5,5,5,5,6,6,6,6,6,6,\\
& 7, 7, 7, 7, 7, 7, 8, 8, 8, 8, 8, 9, 9, 9, 9, 9, 9, 10, 10, 10, 10,
10, 10, 11,  \cdots.\}.
\end{align*}
In fact,
\begin{align*} \set{f(n)}_{n=4985}^{5011}&=\{886,886,886,887,887,887,887,887,887,888,888,888,888,888,889,889,889,\\
&889,889,889,890,890,890,890,890,890,891\}.
 \end{align*}
\end{exam}

\subsection{Quantization for the mixed distribution $P$ when $p=\frac 2{5}$, and  $p=\frac 1{1000}$} \label{sec55}

Let $P_1$ and $P_2$ be two uniform distributions, respectively, on the intervals given by
 \[J_1:=[0, \frac 13], \te{ and } J_2:=[\frac 23, 1].\]
 Let $P:=pP_1+(1-p)P_2$ be the mixed distribution generated by $(P_1, P_2)$ associated with the probability vector $(p, 1-p)$, where $0<p<1$. For $n\in \D N$ , let $\ga_n$ be an optimal set of $n$-means for $P$. Using the similar technique as given in Subsection~\ref{subsec44}, we can show that if $p=\frac 25$, then
 \begin{align*}
 \ga_1=\set{\frac{17}{30}} \te{ with } V_1&=\frac{313}{2700}, \, \ga_2=\set{\frac{1}{6},\frac{5}{6}} \te{ with } V_2=\frac{1}{108}, \,
 \ga_3=\set{\frac{1}{6},\frac{3}{4},\frac{11}{12}} \te{ with } V_3=\frac{11}{2160},\\
 \ga_4&=\set{\frac{1}{12},\frac{1}{4},\frac{3}{4},\frac{11}{12}} \te{ with } V_4=\frac{1}{432}, \te{ and so on}.
 \end{align*}
 On the other hand, if $p=\frac 1{1000}$, then we see that \begin{align*}
 \ga_1&=\set{\frac{1249}{1500}} \te{ with } V_1=0.00970326, \\
 \ga_2&=\set{0.74824116, 0.91608039} \te{ with } V_2=0.0026610135, \\
 \ga_3& =\set{0.719398, 0.831639, 0.94388} \te{ with } V_3=0.00134412,\\
 \ga_4&=\set{0.704407, 0.788862, 0.873317, 0.957772} \te{ with } V_4=0.00087869,\\
 \ga_5&=\set{\frac{1}{6},\frac{17}{24},\frac{19}{24},\frac{7}{8},\frac{23}{24}} \te{ with } V_5=0.000587384 \te{ and so on}.
 \end{align*}

The function $f : \D N \to \D N$, defined in \eqref{eq000000}, is also true here under the condition that $V(j, n-j)$ in this section is defined as follows:
\[
V(j, n-j):=\left\{\begin{array}{ll}
\frac{2}{5} V_j(P_1)+\frac{3}{5} V_{n-j}(P_2) & \te { if } p=\frac 25,\\
\frac 1{1000} V_j(P_1)+\frac{999}{1000} V_{n-j}(P_2) & \te { if } p=\frac 1{1000}
\end{array}
\right. \]
where $1\leq j\leq n-1$. Now, we give the following two examples which are analogous to Example~\ref{exam1} given in the previous section.

\begin{exam} \label{exam2} For $p=\frac 25$, we have
\begin{align*} \set{f(n)}_{n=2}^\infty&=\{1, 1, 2, 2, 3, 3, 4, 4, 5, 5, 6, 6, 7, 7, 7, 8, 8, 9, 9, 10, 10, 11, \
11, 12, 12, 13, 13, 14, 14, \cdots.\}.
\end{align*}
In fact,
\begin{align*} \set{f(n)}_{n=4985}^{5011}&=\{2324,2325,2325,2326,2326,2327,2327,2328,2328,2329,2329,2329,2330,2330,\\
&2331,2331,2332,2332,2333,2333,2334,2334,2335,2335,2336,2336,2336\}.
 \end{align*}
\end{exam}

\begin{exam} \label{exam2} For $p=\frac 1{1000}$, we have
\begin{align*} \set{f(n)}_{n=5}^\infty&=\{1, 1, 1, 1, 1, 1, 1, 1, 1, 1, 1, 2, 2, 2, 2, 2, 2, 2, 2, 2, 2, 2, \
3, 3, 3, 3, 3, 3, 3, 3, 3, 3, 3, 3, 4, 4, \\
&4, 4, 4, 4, 4, 4, 4, 4, 4, 5, 5, 5, 5, 5, 5, 5, 5, 5, 5, 5, 6, 6, 6, 6, 6, 6, 6, 6, 6, 6 \cdots\}.
\end{align*}
In fact,
\begin{align*} \set{f(n)}_{n=4985}^{5011}&=\{453,453,454,454,454,454,454,454,454,454,454,454,454,455,455,\\
& 455,455,455,455,455,455,455,455,455,456,456,456\}.
 \end{align*}
\end{exam}

\begin{remark}
By the results in Subsection~\ref{subsec44} and Subsection~\ref{sec55}, we see that for $p=\frac{1}{100}$ and $p=\frac 1{1000}$ the optimal sets of two-means do not contain any element from $J_1$, but for $p=\frac 25$ it contains an element from $J_1$. Moreover, we see that for $p=\frac 1{100}$ and $p=\frac 25$ the optimal sets of three-means contain elements from $J_1$, but for $p=\frac1{1000}$, the optimal set of three-means, and four-means do not contain any element from $J_1$. Using the similar technique as given in Subsection~\ref{subsec44}, it can be shown that  Lemma~\ref{lemma1111} and  Proposition~\ref{prop2121}, and Theorem~\ref{th3232} are also true for $p=\frac 25$ and $p=\frac{1}{1000}$. The main difference is that for $p=\frac 1{100}$, they are true for all $n\geq 3$, but for $p=\frac 25$, they are true for all $n\geq 2$, on the other hand, for $p=\frac 1{1000}$ they are true for all $n\geq 5$.
\end{remark}
Let us now give the following conjecture and the open problems.

\begin{conj}\label{conj1}
Let $P_1$ and $P_2$ be two uniform distributions defined on any two closed intervals $[a, b]$ and $[c, d]$, where $a<b<c<d$. Let $P:=p P_1+(1-p) P_2$ be a mixed distribution generated by $(P_1, P_2)$ associated with any probability vector $(p, 1-p)$, where $0<p<1$. Then, we conjecture that for each probability vector $(p, 1-p)$ there exists a positive integer $N$ such that for all $n\geq N$, the optimal sets $\ga_n$ contain elements from both the intervals $[a, b]$ and $[c, d]$, and do not contain any element from the open interval $(b, c)$. This yields the fact that if $\te{card}(\ga_n\ii [a, b])=k:=k(n)$, then $\ga_n$ contains $k$ elements from $[a, b]$, and $(n-k)$ elements from $[c, d]$, i.e., $\ga_n(P)=\ga_{k}(P_1)\uu \ga_{n-k}(P_2)$ for all $n\geq N$ with quantization error
\[V_n(P)=\frac{1}{108} \Big(\frac{p}{k^2}+\frac{1-p}{(n-k)^2}\Big).\]
\end{conj}

\begin{open}\label{op1}
Let $P:=pP_1+(1-p)P_2$ be the mixed distribution generated by the two uniform distributions $P_1$ and $P_2$ defined on the closed intervals $[0, \frac 13]$ and $[\frac 23, 1]$ associated with the probability vectors $(p, 1-p)$. It is still not known whether there is any probability vector $(p, 1-p)$, or what is the range of $p$, for which an optimal set $\ga_2$ of two-means for the mixed distributions $P$ will contain an element from the open interval $(\frac 13, \frac 23)$, and an optimal set $\ga_3$ of three-means will contain elements from $[0, \frac 13]$ and $[\frac 23, 1]$, and also from the open interval $(\frac 13, \frac 23)$.
\end{open}

If the answer of the above open problem comes in the negative, then it leads to investigate the following open problem.

\begin{open} \label{op2}
It is still not known whether there is a set of values $a, b, c, d$, and $p$, where $a<b<c<d$ and $0<p<1$, such that if $P:=pP_1+(1-p)P_2$ is the mixed distribution generated by the two uniform distributions $P_1$ and $P_2$ defined on the closed intervals $[a, b]$ and $[c, d]$ associated with the probability vector $(p, 1-p)$, then an optimal set $\ga_2$ of two-means for the mixed distributions $P$ will contain an element from the open interval $(b, c)$, and an optimal set $\ga_3$ of three-means will contain elements from $[a, b]$ and $[c, d]$, and also from the open interval $(b, c)$.
\end{open}

\begin{conj} \label{con11}
We conjecture that the answer of the open problem Open~\ref{op1} will be negative.
\end{conj}

\subsection{Observation} \label{sec66}

Under Conjecture~\ref{con11}, in this subsection, we try to give a partial answer of the open problem `Open~\ref{op2}'. Let us choose the two closed intervals as follows:
\[[a, b]=[0, \frac 7{15}] \te{ and } [c, d]=[\frac 8{15}, 1].\]
Let $P_1$ and $P_2$ be the uniform distributions defined on the closed intervals $[a, b]$ and $[c, d]$. Then, the density functions $f_1$ and $f_2$ for the uniform distributions $P_1$ and $P_2$ are, respectively, given by
\[f_1(x)=\left \{\begin{array} {cc}
\frac{15} 7 & \te{ if } x\in [0,\frac 7{15}], \\
0 & \te{otherwise};
\end{array}
\right. \te{ and } f_2(x)=\left \{\begin{array} {cc}
\frac{15}7 & \te{ if } x\in [\frac 8{15}, 1], \\
0 & \te{otherwise}.
\end{array}
\right.\]
We now give the following two propositions.
\begin{prop1} \label{proo61}
Let $p=\frac{51}{500}$, and let $P:=pP_1+(1-p)P_2$ be the mixed distribution generated by the two uniform distributions $P_1$ and $P_2$ defined on the closed intervals $[a, b]$ and $[c, d]$. Then, an optimal set $\ga_2$ of two-means for the mixed distribution $P$ contains an element from the open interval $(b, c)$.
\end{prop1}

\begin{proof}
Proceeding in the similar way as Proposition~\ref{prop0000}, we see that an optimal set of two-means for the mixed distribution $P:=pP_1+(1-p)P_2$, where $p=\frac{51}{500}$, is given by $\ga_2:=\set{0.488570, 0.829523}$ with quantization error $V_2=0.0179722$. Notice that
\[b<0.488570<c<0.829523<d,\]
and so the assertion of the proposition follows.
\end{proof}

\begin{prop1} \label{proo62}
Let $p=\frac{225}{500}$, and let $P:=pP_1+(1-p)P_2$ be the mixed distribution generated by the two uniform distributions $P_1$ and $P_2$ defined on the closed intervals $[a, b]$ and $[c, d]$. Then, an optimal set $\ga_3$ of three-means for the mixed distribution $P$ contains elements from $[a, b]$ and $[c, d]$, and also from the open interval $(b, c)$.
\end{prop1}

\begin{proof}
Proceeding in the similar way as Proposition~\ref{prop0101}, we see that an optimal set of three-means for the mixed distribution $P:=pP_1+(1-p)P_2$, where $p=\frac{225}{500}$, is given by $\ga_3:=\set{0.174089,  0.522267, 0.840756}$ with quantization error $V_3=0.00985931$. Since
\[a<0.174089<b< 0.522267<c<0.840756<d,\]
and so the assertion of the proposition follows.
\end{proof}

\begin{remark1}
Notice that the two mixed distributions $P:=pP_1 +(1-p)P_2$ considered in Proposition~\ref{proo61} and Proposition~\ref{proo62}, are different. It is worthwhile to investigate whether the two mixed distributions can be same. In other words, whether there is a mixed distribution $P:=pP_1 +(1-p)P_2$, where $P_1$ and $P_2$ are two uniform distributions on two different closed intervals $[a, b]$ and $[c, d]$ associated with a probability vector $(p, 1-p)$ with $a<b<c<d$ and $0<p<1$, for which the open problem `Open~\ref{op2}' is true.
\end{remark1}

\section*{Declaration}

\noindent
\textbf{Conflicts of interest.} We do not have any conflict of interest.\\
\\
\noindent
\textbf{Data availability:} No data were used to support this study.\\
\\
\noindent
\textbf{Code availability:} Not applicable\\
\\
\noindent
\textbf{Authors' contributions:} Each author contributed equally to this manuscript.

\end{document}